\documentclass{article}

\usepackage{amsmath, amsthm, amsfonts, amssymb, amscd,  mathrsfs, mathtools}
\usepackage{latexsym}
\usepackage{amsmath}
\usepackage{amsthm}
\usepackage{amsfonts, subfig, float}
\usepackage{latexsym}
\usepackage{enumerate}
\usepackage{tikz}
\usepackage{adjustbox}
\usepackage{caption}
\usepackage{float}

\usepackage[bookmarks, bookmarksnumbered]{hyperref}
\hypersetup{colorlinks = true, linkcolor = blue, anchorcolor =red, citecolor = blue, filecolor = red, urlcolor = red,  pdfauthor=author}

\usetikzlibrary{calc}

\newtheorem{thm}{Theorem}[section]
\newtheorem{lem}{Lemma}[section]
\newtheorem{definition}{Definition}[section]
\newtheorem{prop}{Proposition}[section]
\newtheorem{proof*}{proof}[section]
\newtheorem{cor}{Corollary}[section]

\newtheorem{claim}{Claim}
\usepackage{indentfirst}
\numberwithin{equation}{section}

\newcommand{\CA}{{\mathcal{A}}}

\newcommand{\CP}{{\mathcal{P}}}

\newcommand{\CC}{{\mathcal{C}}}

\newcommand{\x}{\mathbf{x}}

\DeclareMathOperator{\arcosh}{arcosh} 
\DeclareMathOperator{\sech}{sech} 

\title{The smallest spectral radius of bicyclic uniform  hypergraphs with a given size}

\author{
Haiying Shan\thanks{\footnotesize School of Mathematical Sciences, Tongji University, Shanghai 200092, China
(\texttt{shan\_haiying@tongji.edu.cn})},~
Zhiyi Wang\thanks{School of Mathematical Sciences, Tongji University, Shanghai 200092, China
(\texttt{1610519@tongji.edu.cn})},~
Feifei Wang\thanks{School of Mathematical Sciences, Tongji University, Shanghai 200092, China
(\texttt{1710854@tongji.edu.cn})}
}

\date{}

\begin{document}
\maketitle

  \begin{abstract}
    Identifying graphs with extremal properties is an extensively studied topic in spectral graph theory.
    In this paper, we study the log-concavity of a type of iteration sequence related to the $\alpha$-normal weighted incidence matrices which is presented by Lu and Man for computing the spectral radius of hypergraphs. By using results obtained  about the  sequence and the method of some edge operations, we
    will characterize completely extremal $k$-graphs with the smallest  spectral radius among bicyclic hypergraphs with given size.

    \bigskip


  \end{abstract}

  \section{Introduction}

  
  A hypergraph $H=(V, E)$ consists of a set $V$ of {\it vertices} and a collection $E$ of subsets of $V$.
  We call a member of $E$ a {\it hyperedge} or simple an {\it edge} of $H$. A hypergraph $H$ is
    {\it $k$-uniform} if all of its edges are $k$-subsets of $V(H)$. We also simply call a $k$-uniform hypergraph
  a {\it $k$-graph} for brevity.
  Throughout this paper,  we focus on simple $k$-uniform hypergraphs with   $k \geq  3$.

  A hypergraph $H$ is a sub-hypergraph of $G$ if $V (H) \subseteq V(G)$ and $E(H) \subseteq E(G)$,  and $H$ is a proper sub-hypergraph of $G$ if $V(H) \subsetneq V(G)$ or $E(H) \subsetneq  E(G)$.
  
  Let $v$ be a vertex of $H$. We use $E_H(v)$ to denote the set consisting of edges of $H$ which are incident to $v$, also written as $E(v)$ for brevity.
  The cardinality of  $E_H(v)$ is defined as the degree of the vertex $v$,  denoted  $d_H(v)$ or $d(v)$. A vertex of degree one is called a pendant vertex. For $e=(v_1, v_2, \ldots, v_k) \in E(H)$, if there exists at least $k-1$ pendant vertices, we say $e$ is a pendant vertex of $H$; if there exists  at least one vertex of degree greater than 2 or three vertices of degree greater than 1, we say  $e$ is a branch edge of $H$.

  In hypergraph theory,  the most common definition of a cycle in a hypergraph is given by Berge (see \cite{Berge1973}).

  A {\it Berge cycle of length $\ell$} is a hypergraph $\CC$ consisting of
  $\ell$ distinct hyperedges $e_1, ..., e_\ell$ such that there exist $\ell$ distinct vertices $v_1, \dots,  v_\ell$ satisfying that
  $v_i, v_{i+1}\in e_i$ for each $i=1, ..., \ell-1$ and $v_1, v_\ell\in e_\ell$.
  
  A {\it Berge path of length $\ell$}  is a hypergraph $\CP$ consisting of $\ell$ distinct edges $e_1, \dots,  e_\ell$
  such that there exist $\ell+1$ distinct vertices $v_1, \dots,  v_{\ell+1}$ satisfying that
  $v_i,  v_{i+1}\in e_i$ for $i=1, \ldots,  \ell$.
  
  In graph theory, an internal path of a graph $G$ is a sequence of vertices $u_1,\ldots,u_n$ such that all $u_i$ are distinct, except possibly $u1 = u_n$; the vertex degrees satisfy
  $d(u_{1}) \geq 3, d(u_{2}) = \cdots = d(u_{n-1}) = 2, d(u_{n}) \geq  3$, and $u_{i}u_{i+1} \in E(G)$ for $i=1, \ldots, n - 1$.

  Let $G = (V, E)$ be a simple graph. The $k$-th power hypergraph of $G$ is the $k$-uniform hypergraph resulting from adding $k-2$ new vertices to each edge of $G$.
  For the ordinary path  and cycle with $n$ edges,  the $k$-th power of them are called loose path and loose cycle with length $n$,  denoted by $\mathbb{P}_n$ and $\mathbb{C}_n$,  respectively.

  In the following, we will  give a general definition of internal path for hypergraph.
  
  \begin{definition}For hypergraph $H$, 
    let $P=e_0v_1e_1v_2 \cdots e_{n-1}v_{n}e_n$ be a vertex edge alternating sequence of $H$ such that
    $\mathbb{P}=v_1e_1v_2 \cdots e_{n-1}v_{n}$ is a loose path of hypergraph $H$ and $d_H(v_1)=d_H(v_n)=2$, $\{e_0, e_n\}\subseteq E(H)$  (not necessarily distinct).
    \begin{enumerate}[(1).]
      \item If $e_0$ and $e_n$ are both branch edges, then $P$ is called hyper-internal path of $H$ and $\mathbb{P}$ is called internal loose path.
      \item If $P$ is a hyper-internal path and one of $e_0$ and $e_n$ contains no vertex with degree more than 2, then $P$ is called hypo-internal path of $H$.
    \end{enumerate}
  \end{definition}
  
  In graph theory,  the cyclomatic number of connected graph $G$ is $|E(G)|-|V(G)|+1$,  which corresponds to the number of independent (in the sense of linear independence in the so-called cycle space) cycles in $G$.  In \cite{Fan2016a},  Fan  et al.  present a generalization of cyclomatic number for $k$-uniform graph.
  Let $H$ be a $k$-graph with $n$ vertices,  $m$ edges,  and $p$ connected components. The cyclomatic number of $H$ is denoted as $c(H)$ and defined as $c(H) = m(k-1)-n + p$.  The connected hypergraph $H$ with $c=c(H)$ is called a $c$-cyclic hypergraph.
  In particular,   $0$-cyclic hypergraph,   $1$-cyclic hypergraph,   $2$-cyclic hypergraph,  are called supertree,  unicyclic $k$-graph,  bicyclic $k$-graph,  respectively.

  A very important topic on spectral hypergraph theory is the characterization of hypergraphs with extremal values of the spectral radius in a given class of hypergraphs. The spectral radii of hypertrees are well studied in the literatures,   see \cite{Li2016a, xiao2017maximum, Yuan2016a}.  In \cite{lu2016connected},  Linyuan Lu and Shoudong Man research the connected hypergraphs with small spectral radii   and present the $\alpha$-normal labeling method  for  comparing the spectral radii of connected $k$-graphs.

  In \cite{Fan2016a,  Ouyang2017},  hypergraphs that attain the largest spectral radii among all unicyclic and bicyclic k-graphs are investigated.
  Recently,  Zhang et al. determine the uniform hypergraphs of  size $m$ with the first two smallest spectral radii \cite{Zhang_2020}.
  
  Motivating by the preceding work on maximizing and ordering spectral radius,  we will consider spectral extremal problems for bicyclic hypergraphs with given size in this paper.

  The remainder of this paper is organized as follows: In Section 2,  we give some basic definitions and results for tensor and spectra of hypergraphs. In Section 3,  a kind of iteration sequence of m\"obius transformation,  which related with some weighted incidence matrix of hypergraph,  is researched.  We present an explicit expression of the sequence and some property on the convexity  and log-concavity of the sequence is researched. In Section 4,  effects on the  spectral radius of hypergraph by graph operations are researched.  In Section 5,  the first smallest spectral radius of bicyclic hypergraph are determined.

  \section{Preliminaries}

  The definition of eigenvalues of a tensor was proposed in \cite{Lim2005, Qi2005},  independently. Let $\CA$ be an order $k$ dimension $n$ tensor.  If there exists a number $\lambda \in \mathbb{C}$ and a nonzero vector $\x = (x_1, x_2, ..., x_n)^T \in  \mathbb{C}^n$  such that
  $$  \CA \x = \lambda \x^{[r - 1]}, $$
  where $\x^{[r- 1]}$ is a vector with $i$-th entry $x_{i}^{r - 1}$,  then $\lambda$  is called an eigenvalue of $\CA$,  $\x$ is
  called an eigenvector of $\CA$ corresponding to the eigenvalue $\lambda$.  The spectral radius of $\CA$ is the maximum modulus of the eigenvalues of $\CA$,  denoted by $\rho(\CA)$.
  
  In 2012,  Cooper and Dutle \cite{cooper2012spectra} defined the adjacency tensors for $k$-uniform hypergraphs.
  
  \begin{definition}[\cite{cooper2012spectra}]\label{def3}
    The adjacency tensor $\mathcal A(G)=(a_{i_1\dots i_k})$ of a $k$-uniform hypergraph $G$ is defined to be an order $k$ dimension $n$  tensor with entries $a_{i_1\dots i_k}$ such that $$a_{i_1\dots i_k}=
      \begin{cases}
        \frac{1}{(k-1)!} & \text{if $\{i_1, i_2, \dots, i_k\} \in E(G)$}, \\
        0                & \text{otherwise}.
      \end{cases}$$
  \end{definition}

  The following result can be found in \cite{cooper2012spectra, Khan2015} and will be used in the sequel.
  \begin{thm}\label{proper:subgraph}
    Suppose that $H$ is a uniform hypergraph,  and $H'$ is a sub-hypergraph of $H$. Then $\rho(H') \leq  \rho(H)$. Furthermore,  if in addition $H$ is connected and $H'$ is a
    proper sub-hypergraph,  we have $\rho(H') < \rho(H)$.
  \end{thm}

  \begin{definition}\label{def1}\cite{lu2016connected}
    A weighted incidence matrix $B$ of a hypergraph $H$ is a
    $\mid V\mid \times \mid E \mid $ matrix such that for any vertex $v$ and edge $e$,  the entry $B(v, e) > 0 $ if $v\in e $ and $B(v, e)= 0$ if $v\notin e$.
  \end{definition}

  Let  $\CC_m=v_0e_1v_1e_2\cdots e_mv_m$ ($v_0=v_m$) be a Berge cycle of $H$.  For a weighted incidence matrix $B$ of $H$,  if
  $$\prod_{i=1}^{m}\frac{B(v_i, e_i)}{B(v_{i-1}, e_i)}=1, $$ we say that $B$ is consistent with $\CC_m$.  If $B$ is consistent for any Berge cycle of $H$,  we say that $B$ is consistent for $H$.

  \begin{definition}\cite{lu2016connected}
    A hypergraph $H$ is called $\alpha$-supernormal if there exists a weighted incidence matrix B satisfying
    \begin{enumerate}[(1).]
      \item{$\sum\limits_{e \in E(v)}B(v, e) \geq 1$,  for any $v \in V(H) $, }
      \item{$\prod\limits_{v\in e}B(v, e) \leq \alpha$,  for any $e \in E(H)$.}
    \end{enumerate}
    If both equalities hold in above inequalities,  $H$ is called $\alpha$-normal. Otherwise,  $H$ is called strictly $\alpha$-supernormal. Furthermore,
    if $B$ is consistent, $H$ is called  consistently $\alpha$-supernormal.
  \end{definition}

  \begin{definition}\label{def2}\cite{lu2016connected}
    A hypergraph $H$ is called $\alpha$-subnormal if there exists a weighted incidence matrix B satisfying
    \begin{enumerate}[(1).]
      \item{$\sum\limits_{e \in E(v)}B(v, e) \leq 1$,  for any $v \in V(H) $,}\label{tem1}
      \item{$\prod\limits_{v\in e}B(v, e) \geq \alpha$,  for any $e \in E(H)$.}\label{tem2}
    \end{enumerate}
    If one of above inequalities is strict,  $H$ is called strictly $\alpha$-subnormal.
  \end{definition}
  
  \begin{thm}\label{lem_lu}\cite{lu2016connected}
    Let $H$ be a connected $k$-uniform hypergraph. Then the following results hold:
    \begin{enumerate}[(1).]
      \item $\rho$ is the spectral radius of $H$ if and only if $H$ is consistently $\alpha$-normal with $\alpha=\rho^{-k}$.
      \item If $H$ is $\alpha$-subnormal,  $\rho\leq \alpha^{-\frac{1}{k}}$. Moreover,  if $H$ is strictly $\alpha$-subnormal,  $\rho < \alpha^{-\frac{1}{k}}$.
      \item  If $H$ is strictly and consistently $\alpha$-supernormal,  $\rho >\alpha^{-\frac{1}{k}}$.
    \end{enumerate}
  
  \end{thm}

  \begin{lem}\label{lem:weighted_matrix}
    Let $H$ be a connected $k$-uniform hypergraph,  and $\x$ be the principal eigenvector of $H$. Define the weighted incidence matrix $B$ such that
    $$B(v, e)= \begin{cases}
        \displaystyle (\prod_{i \in e}\frac{x_i}{x_v})\rho(H)^{-1} & \text{ if }  v \in e, \\
        0                                                          & \text{otherwise.}
      \end{cases}$$
  
    Then $H$ is consistently $\alpha$-normal,  where $\alpha(H)=\rho(H)^{-k}$.
  \end{lem}
  
  In the sequel,  the weighted incidence matrix of hypergraph $H$ defined in Lemma \ref{lem:weighted_matrix} is denoted by $B_H$.

  \begin{definition}
    Let $\mathbb{P}$ be an internal path of length $n$ between $w_0$ and  $w_n$  in hypergraph  $H$. Let $e_1, e_n\in E(\mathbb{P})$ and $w_0 \in e_1,  w_n \in e_n$,   $B$ is a weighted incidence matrix  of $H$. Write $x=B(w_0, e_1), y=B(w_n, e_n)$.
  
    We say that $B$ is $(x, y)$ $\alpha$-normal on $\mathbb{P}$,  if  $B$ satisfies the following conditions:
    \begin{enumerate}[(1).]
      \item   $\sum\limits_{e \in E(v)}B(v, e) = 1$,  for any $v \in V(\mathbb{P})\setminus \{w_0, w_n\} $,
      \item $\prod\limits_{v\in e}B(v, e) = \alpha$,  for any $e \in E(\mathbb{P})$.
    \end{enumerate}
  \end{definition}

  \begin{lem}\label{prop-consistent}
    Let $e'=(u_1, u_2, \dots, u_k),  e''=(v_1, v_2, \dots,  v_k)$ be two  edges of $k$-uniform hypergraph $H$ and $B$ be a weighted incidence matrix of $H$.
    \begin{enumerate}[(1).]
      \item Let $P_1$ is a $(u_1, u_2)$- internal loose path of $H$.  If $B$ is $(x_0, x_0)$ $\alpha$-normal for $P_1$ and
            $B(u_1, e')=B(u_2, e')=1-x_0$,
            then $B$ is consistent for cycle $u_1P_1u_2e'u_1$.
      \item
            Let $P_1$ is a $(u_1, v_1)$- internal loose path and  $P_2$ is a $(u_2, v_2)$- internal loose path  of $H$. For $i=1,2$,   if $B$ is $(y_i, y_i)$ $\alpha$-normal for $P_i$ and $B(u_i, e')=B(v_i, e'')=1-y_i$,
            then  $B$ is consistent for cycle $u_1P_{1}v_1e''v_2P^{-1}_{2}u_2e'u_1$.
    \end{enumerate}
  \end{lem}
  
  \begin{proof}
    (1). Suppose
    $P_{1}=w_0e_1w_1e_2w_2 \cdots e_{n-1}w_{n-1}e_nw_{n} $  and $w_0=u_1, w_{n}=u_2$. Let
    $x_i=B(w_i, e_{i+1}), x'_i=B(w_{i},e_{i})$ ($i=1, 2, \ldots, n-1$) and $x_0=B(w_0,e_1)$, $x'_{n}=B(w_n,e_n)$. Since $B$ is $(x_0, x_0)$ $\alpha$-normal for $P_1$, we have $x'_{n-j}=x_j$ for any $j \in \{0,1,\ldots,n\}$.
  
    For Berge cycle $u_1P_1u_2e'u_1=w_0e_1w_1e_2w_2 \cdots e_{n-1}w_{n-1}e_nw_{n}e'w_0$, we have
    $$\frac{B(w_0,e')}{B(w_n,e')}\prod_{i=1}^{n}\frac{B(w_i, e_i)}{B(w_{i-1}, e_i)}=1.$$
    So $B$ is consistent for cycle $u_1P_1u_2e'u_1$.
  
    (2).
    Using similar arguments the result can be deduced  from the conditions.
  \end{proof}

  \section{On a kind of sequence obtained by iteration of Möbius transform}

  In this section,  we will focus on the convexity and log-concavity of some sequences which are related with the $\alpha$-normal weighted incidence matrices of hypergraphs.
  Proposition \ref{prop32} of this section will play a key role  in the proof of our main result (Theorem \ref{mainresult}) of this paper.

  Let $H$ be a connected $k$-graph and
  $\mathbb{P}_{n}=w_0e_1w_1e_2w_2 \cdots e_{n-1}w_{n-1}e_nw_{n} $ be a loose path in $k$-graph $H$. Let $B$ be a $\alpha$-normal weighted incidence matrix. Take
  $x_i=B(w_i, e_{i+1})$ ($i=0, 1, \dots, n-1$) and $\displaystyle x_n=\sum_{e \in E_H(v)\setminus \{e_n\}}B(v, e)$. Then we have:
  $x_i(1-x_{i+1})=\alpha$  for $i=0, \dots, n-1$.
  
  Hence,
  $(x_i)_{i=0}^{n}$ is finite portion of the following bi-infinite sequence $(x_i)_{i\in \mathbb{Z}}$ such that
  \begin{equation}\label{eq-rec}
    x_n=f(x_{n-1})=f^n(x_0) \text{ for } n\in \mathbb{Z},
  \end{equation}
  where $f(x)=1- \frac{\alpha}{x}$,  which is a special type of M\"{o}bius transform (see Section 1.2 of \cite{beardon2000iteration}).

  $\phi(x)=x^2-x+\alpha=0$ is called characteristic equation of the iterative sequence $(x_i)$. The roots of $\phi(x)=0$ are called the characteristic roots of the sequence $(x_i)$.
  
  When $0<\alpha<\frac{1}{4}$,  $\phi(x)$ have two distinct positive roots,  denoted by $r_1, r_2$. Suppose $r_2 < r_1$.
  Let $\theta=\arcosh(\frac{1}{2} \alpha^{-\frac{1}{2}})$.  Then $r_1, r_2$ can be written in the following form:
  $$r_1=\sqrt{\alpha} e^{\theta}, \qquad  r_2=\sqrt{\alpha} e^{-\theta}.$$
  
  Since $1-4\alpha=1- sech^2(\theta)=\tanh^2(\theta)$,  $r_1, r_2$ can be rewritten as:
  $$r_2 =\frac{1}{2}(1- \tanh(\theta)),  \qquad r_1= \frac{1}{2}(1+\tanh(\theta)).$$
  
  For sequence \eqref{eq-rec},  if there exists some integer $j$ such that $x_j=\alpha$,  then $x_{j+1}=0$,  $x_{j+2}=-\infty$ and $x_{j+3}=1$.
  
  If the sequence \eqref{eq-rec} is positive sequence and there exist integers $p,  q$ such that $x_p+x_q=1$,  we call that  $(x_i)$ is symmetric.
  
  It is obvious that $(x_i)$ is symmetric when $x_0= \frac{1}{2},  \alpha=\frac{1}{4}$.

  \begin{lem}\label{thm-recur}
    Let $(x_i)_{i\in \mathbb{Z}}$ be the sequence mentioned above.
    \begin{enumerate}[(1).]
      \item  If
            $\alpha \in (0, \frac{1}{4}]$ and $x_0 \in (0, 1]$,   then  $$x_n=\frac{2\cosh(\theta)\sinh((n+1)\theta)x_0 -\sinh(n\theta)}{2\cosh(\theta)\big(2\cosh(\theta)\sinh(n\theta)x_0 -\sinh((n-1)\theta)\big)}, $$ where $\theta=\arcosh(\frac{1}{2} \alpha^{-\frac{1}{2}})$.
      \item   If $\phi(x_0)<0$,  then $(x_i)_{i\in \mathbb{Z}}$ is strictly increasing and  $\displaystyle \lim_{i\to \infty}x_i=r_1$.
      \item If $x_0>r_1$,  then $(x_i)_{i\in \mathbb{Z}}$ is strictly decreasing and  $\displaystyle \lim_{i\to \infty}x_i=r_1$.
      \item If $x_0<r_2$,  then there exists $n_0>0$ such that $x_{n_0}\leq 0$ and $(x_i)_{i=0}^{n_0-1}$ is positive decreasing sequence.
    \end{enumerate}
  \end{lem}
  
  \begin{proof}
    (1).  When $\alpha= \frac{1}{4}$,  $\theta=0$,  for any integer $n$,  $x_n= \frac{1}{2}.$ The result holds.

    When $\alpha \neq \frac{1}{4}$,  as mentioned above,  $r_1,  r_2$ be characteristic roots of sequence $(x_n)_{n \in \mathbb{Z}}$.
    By the formula (8) of \cite{brand1955sequence},  we have:
    $$
      \begin{aligned}
        x_n= & \frac{ (r_1^{n+1}-r_2^{n+1})x_0  - r_1r_2(r_1^n-r_2^n)}{(r_1^n-r_2^n)x_0  - r_1r_2(r_1^{n-1}-r_2^{n-1}) }                                \\
        =    & \frac{2\cosh(\theta)\sinh((n+1)\theta)x_0 -\sinh(n\theta)}{2\cosh(\theta)\big(2\cosh(\theta)\sinh(n\theta)x_0 -\sinh((n-1)\theta)\big)}.
      \end{aligned}
    $$
  
    (2). Since $\phi(x_i)<0$,  $\phi(x)=0$ have two distinct positive roots,  denoted by $r_1,  r_2$ as before. So $r_2<x_i<r_1$.
    Because $f(x),  f^{-1}(x)$ are strictly increasing functions in the interval $(0,  \infty)$,  we have
    $$r_2=f(r_2)<f(x_i)<f(r_1)=r_1, $$ $$r_2=f^{-1}(r_2)<f^{-1}(x_i)<f^{-1}(r_1)=r_1.$$
    Therefore,  $r_2<x_{i-1}<r_1$ and $r_2<x_{i+1}<r_1$.
  
    So $\forall j \in \mathbb{Z}$,  $r_2<x_{j}<r_1$,  $\phi(x_j)<0$.  $x_j-f(x_j)<0$ follows.
  
    Since $x_{j+1}=f(x_j)$,  we have  $x_j<x_{j+1}$ for any $j \in \mathbb{Z}$.
  
    Hence,   $(x_i)$ is strictly increasing bounded sequence.
    And
    $\displaystyle \lim_{i\to +\infty}x_i$ exists. Suppose $\displaystyle \lim_{i\to +\infty}x_i=z$. It is clear that $z>r_2$. Next,  we take limits as $i\to \infty$ on both sides of  $x_{i+1}=f(x_i)$. We have
    $z=1- \frac{\alpha}{z}$. So $z$ is a root of $\phi(x)=0$.
    Since $z>r_2$,  $z=r_1$ follows.  That is,  $\displaystyle \lim_{i\to \infty}x_i=r_1$.

    (3)  If $x_i > r_1$,  then $\phi(x_i)>0$ and $x_i> f(x_i)=x_{i+1}>r_1$. So $(x_i)_{i=0}^{+\infty}$  is a decreasing bounded positive  sequence. So $\displaystyle \lim_{i\to \infty} x_i$ exists,  denoted by $r$. Furthermore,  $r$ will be a root of $\phi(x)=0$. From $x_0>r_1$,  we have $r \geq r_1$. Since $r_1$ is the maximal root of $\phi(x)$,  we have $r=r_1$.

    (4) If $0<x_i <r_2$,  then $\phi(x_i)>0$ and $x_i> f(x_i)=x_{i+1}$. There must exist $n_0>0$ such that $x_{n_0}\leq 0$. Otherwise,   $(x_i)_{i=0}^{+\infty}$  is a  bounded positive sequence. So $\displaystyle \lim_{i\to \infty} x_i$ exists,  denoted by $r$. Furthermore,  $r$ will be a root of $\phi(x)=0$. From $x_0<r_2$,  we have $r<r_2$,  which contradicts that $r_2$ is the minimal root of $\phi(x)$. The result follows.\end{proof}
  
  Let
  $$
    \begin{aligned}
      F(n, x_0)= & \frac{2\cosh(\theta)\sinh((n+1)\theta)x_0 -\sinh(n\theta)}{2\cosh(\theta)\big(2\cosh(\theta)\sinh(n\theta)x_0 -\sinh((n-1)\theta)\big)}.
    \end{aligned}
  $$
  For the sake of brevity,  take $p(n)=2\cosh(\theta)\sinh((n+1)\theta)x_0-\sinh(n\theta).$
  Then $$F(n, x_0)= \frac{p(n)}{2\cosh(\theta)p(n-1)}.$$

  Suppose the sequence $(x_i)_{i \in \mathbb{Z}}$ is symmetric. Without loss of generality,    suppose $x_p+x_q=1$ ($p= q+l \in \mathbb{Z}$,  $l \geq 0$).  Take $y_i=x_{i-q}$. Then $(y_i)_{i \in \mathbb{Z}}$ is an iteration sequence of $f(x)=1- \frac{\alpha}{x}$ with initial value $y_0$. By (1) of Lemma \ref{thm-recur},  we have $y_n=F(n, y_0)$ for all $n \in \mathbb{Z}$.
  
  For $x_p+x_q=1$,  $y_0+y_l=1$. So $y_0 + F(l, y_0)=1$. Namely,  $y_0$ is a root of the equation $x+F(l, x)=1$.
  
  Furthermore,  we have the following conclusion:
  \begin{thm}\label{sym_rec}
    Let $(y_i)_{i\in \mathbb{Z}}$ be the symmetric sequence mentioned as above.
    \begin{enumerate}[(1).]
      \item For any $s \in \mathbb{Z}$,    we have $y_{l+s}+y_{-s}=1$.
      \item If $0<\alpha<\frac{1}{4}$,  then  $(y_i)_{i\in \mathbb{Z}}$ is strictly increasing bounded sequence with  $r_2 <y_i < r_1$ for any $i$ and $\displaystyle  \lim_{i\to +\infty} y_i= r_1$.
      \item For $0<\alpha<\frac{1}{4}$,
            $y_0$ is the largest root of $x+F(l, x)=1$ and
            \[
              y_0=\frac{1}{2}\bigg(1- \tanh(\theta)\tanh\left(\frac{l\theta}{2}\right)\bigg).
            \]
    \end{enumerate}
  
  \end{thm}

  \begin{proof}
    (1). It is sufficient to show that if there exist integers $p, q$ such that $y_p+y_q=1$,  then
    $y_{p+1}+y_{q-1}=1$.
  
    For $f(x)=1- \frac{\alpha}{x}$,  we have \\
    $y_{p+1}=f(y_p)=1- \frac{\alpha}{y_p},  y_{q-1}= f^{-1}(y_q)= \frac{\alpha}{1-y_q}= \frac{\alpha}{y_p} $. So $y_{p+1}+y_{q-1}=1$. \\
    Since $y_0+y_l=1$,  $y_{l+s}+y_{-l}=1$ holds for any $s \in \mathbb{Z}$.

    (2). Take $t=\lfloor \frac{l}{2}\rfloor$. From (1),
    $y_{l-t}+y_t=1$ holds.  So
    $$y_{t}=  \begin{cases}
        \sqrt{\alpha} & \text{ if $l$ is odd}, \\
        \frac{1}{2}   & \text{otherwise}.
      \end{cases}
    $$
    Hence,
    $r_2< \sqrt{\alpha}< \frac{1}{2}<r_1$ and $\phi(y_t)<0$.
  
    The result follows from (2) of Lemma \ref{sym_rec}.

    (3). Since $y_0$ is a root of $x+F(l, x)=1$,  By (1) of Lemma \ref{thm-recur},  $y_0$ is a root of the following equation:
    $$\frac{2\cosh(\theta)\sinh((l+1)\theta)x-\sinh(l\theta) }{2\cosh(\theta)\big(2\cosh(\theta)\sinh(l\theta)x- \sinh((l-1)\theta)\big)}+x=1, $$
    which has the same roots as the  following quadratic equation
    $$ x^2 - \frac{\sinh((l-1)\theta)}{\sinh(l\theta)\cosh(\theta)} x + \frac{\sinh((l-2)\theta))}{4\sinh(l \theta)\cosh^2(\theta)} =0.$$
  
    Let $\Delta$ be the discriminant of above quadratic equations. Then
    $$\sqrt{\Delta}
      = \frac{\sqrt{\sinh^2((l-1)\theta)- \sinh(l\theta)\sinh((l-2)\theta)}}{\sinh(l\theta) \cosh(\theta)}= \frac{\sinh(\theta)}{\sinh(l\theta)\cosh(\theta)}.$$
    By the quadratic formula we obtained the two roots of above equation:
    $$\nu_0=\frac{\sinh((l-1)\theta) -\sinh(\theta)}{2\sinh(l\theta)\cosh(\theta)},  \quad
      \nu_1=\frac{\sinh((l-1)\theta) + \sinh(\theta)}{2\sinh(l\theta)\cosh(\theta)}.$$
    Next,  we will complete the proof of (2) by showing that $y_0=\nu_1$.
  
    Since $r_2<y_0<r_1$,  to show $y_0=\nu_1$,  it is efficient to prove that $\nu_0 < r_2$.
  
    Since $r_2 =\frac{1}{2}(1- \tanh(\theta))$,  we have
    $$
      \begin{aligned}
             & \nu_0<r_2                                                                          \\
        \iff & \sinh((l-1)\theta) - \sinh(\theta) < (\cosh(\theta) - \sinh(\theta))\sinh(l\theta) \\
        \iff & -\sinh(\theta) < (\cosh(l\theta)-\sinh(l\theta))\sinh(\theta).
      \end{aligned}
    $$

    Since  $\theta=\arcosh(\frac{1}{2} \alpha^{-\frac{1}{2}})>0$,  the last condition holds. So
    $$ y_0=\nu_1=\frac{\sinh((l-1)\theta) + \sinh(\theta)}{2\sinh(l\theta)\cosh(\theta)}.$$
    This formula can be rewritten as
    \[
      y_0=\frac{1}{2}(1- \tanh(\theta)\tanh(\frac{l\theta}{2})).   \qedhere
    \] \end{proof}
  In the sequel,  we always write
  $$F_0(x)=\frac{1}{2}(1- \tanh(\theta)\tanh(\frac{x\theta}{2}))$$ and  $$F_0^*(x)=1-F_0(x)=\frac{1}{2}(1 + \tanh(\theta)\tanh(\frac{x\theta}{2})).$$
  Consequently,  $y_0=F_0(l)=F_0^*(-l)$ and $F_0(x)=F_0^*(-x)$.
  
  From (1) and (3) of Theorem \ref{sym_rec},  we have
  \begin{equation}\label{eq:sym}
    y_{-s}=F_0(l+2s).
  \end{equation}

  Furthermore,  we have the following Corollary.
  \begin{cor}\label{cor3.1}
    $(x_i)_{i \in \mathbb{Z}}$ is the symmetric sequence defined as \eqref{eq-rec}.  Without loss of generality,    suppose $x_p+x_q=1$ ($l\geq 0,  p = q+l \in \mathbb{Z}$). Then\\
    for $\alpha \in (0, \frac{1}{4}]$,  $x_n$ can be expressed as:
    $$x_n=F^*_0(2n-p-q).$$
  \end{cor}
  \begin{proof}
    Take $y_i=x_{i-q}$. Then $(y_i)_{i\in \mathbb{Z}}$ is the iteration sequence of $f(x)=1- \frac{\alpha}{x}$ with initial value $y_0$.
    Since $x_p+x_q=1$,  $y_0+y_l=1$.
  
    By \eqref{eq:sym},  we have
    \[x_n=y_{n-q}=F_0(l-2(n-q))=F_0(p+q-2n)=F_0^*(2n-p-q).   \qedhere  \]
  \end{proof}
  
  The remainder of this section is dedicated to the study of the convexity of $F_0(x)$ and the log-concavity  of $F_0^*(x)$.

  
  \begin{definition}
    A function $f$ is log-concave on $I \subseteq \mathbb{R}$,  if $f(x)\geq 0$ for all $x \in I$  and the function $x\mapsto log f(x)$ is concave on $I$.
  \end{definition}

  It is well known (see \cite{marshall1979inequalities}) that
  
  \begin{thm}\label{schur-concave}
    Let  $F$ be positive function on  $I \in \mathbb{R}$. Then the function
    $$
      \varphi\left(x_{1},  \ldots,  x_{n}\right)=\prod_{i=1}^{n} F\left(x_{i}\right)
    $$
    is Schur-concave if and only if $F$ is log-concave.
  
  \end{thm}

  \begin{thm}\label{thm:maxroot}
    Let $F_0(x)$ and $F_0^*(x)$ be as defined above, then
    \begin{enumerate}[(1).]
      \item $F_0(x)$ is strictly decreasing and convex in the interval $(0, +\infty)$;
      \item   $F_0^*(x)$ is strictly increasing and concave in   $(0, +\infty)$  and log-concave in  $[-1, +\infty)$.
    \end{enumerate}
  \end{thm}
  
  \begin{proof}
    (1).
    For $F_0(n)= \frac{1}{2}(1- \tanh(\theta)\tanh(\frac{n\theta}{2}))$,  by simple computation,  we have
    $$F_0(x)'=-\frac{\theta}{4} \tanh(\theta) \sech^2(\frac{x\theta}{2})$$
    and
    $$F_0(x)''=\frac{\theta^2}{4} \tanh(\theta) \tanh(\frac{x\theta}{2})\sech^2(\frac{x\theta}{2}).$$
  
    Since $\theta>0$,  $F_0(x)'<0$ and $F_0(x)''>0$ on $(0, +\infty)$. So $F_0(x)$ is strictly decreasing and convex in the interval $(0, +\infty)$.
  
    (2).
    Since $F_0^*(x)=1- F_0(x)$,  from (1),  it follows that  $F_0^*(x)$ is strictly increasing and concave in the interval  $(0, +\infty)$.
  
    Since
    \[\begin{aligned}
        F^*_0(x) & = \frac{\sinh((x+1)\theta) - \sinh(\theta)}{2\sinh(x\theta)\cosh(\theta)} \\
                 & = \frac{1}{2}(1+\tanh(\theta)\tanh(\frac{x\theta}{2})).
      \end{aligned} \]
  
    Let \(f^*(x)=log(F^*_0(x))\). We have
    \[f^*(x)'=\frac{\theta  {(\cosh(\theta  x) - 1)} \sinh(\theta )}{{(\sinh(\theta  {(x + 1)}) - \sinh(\theta ))} \sinh(\theta  x)}= \frac{\theta \sinh(\theta ) \tanh(\frac{x \theta }{2}) }{{(\sinh(\theta  {(x + 1)}) - \sinh(\theta ))} }\]
    and
    \[ f^*(x)''=-\frac{(\theta \tanh(\frac{1}{2}  \theta x))^2\sinh(\theta {(x + 1)}) \sinh(\theta)  }{(\sinh(\theta {(x + 1)}) - \sinh(\theta))^2}. \]
  
    For $\theta>0$,  $ x > -1$,  we have $f^*(x)''<0$.  So  $F_0^*(x)$ is  log-concave in  $[-1, +\infty)$.
  \end{proof}
  
  \begin{prop}\label{prop32}
    Suppose  $a, b, c, d$ are real numbers such that $0 \leq a <b \leq c < d$, then
    \begin{enumerate}[(1).]
      \item If $a+d=b+c$,  we have
            \(F_0^*(a)F_0^*(d)<F_0^*(b)F_0^*(c)\);
      \item $F_0^*(b)F_0^*(-a)<F_0^*(0)F_0^*(b-a)$.
    \end{enumerate}
  \end{prop}
  \begin{proof}
    (1).
    From (2) of  Theorem \ref{thm:maxroot} and Theorem \ref{schur-concave},  the assertion (1) holds.
  
    (2). From $F_0^*(-x)+F_0^*(x)=1$ and $0 \leq F_0^*(-x), F_0^*(x)<1$,  we have
    $F_0^*(a)F_0^*(-a) \leq (\frac{F_0^*(a)+F_0^*(-a)}{2})^2=\frac{1}{4}$.
  
    Using (2) of  Theorem \ref{thm:maxroot} and Theorem \ref{schur-concave},
    we have\\
    $F_0^*(b)=2F_0^*(0)F_0^*(b)<2F_0^*(b-a)F_0^*(a)$.
  
    So
    \[
      \begin{aligned}
        F_0^*(b)F_0^*(-a) & <2F_0^*(b-a)F_0^*(a)F_0^*(-a) \\ &<\frac{1}{2}F_0^*(b-a)=F_0^*(0)F_0^*(b-a).
      \end{aligned}
    \]
    The assertion (2) holds. \end{proof}

  \section{The effect on the spectral radii of uniform hypergraphs by some graph operations}
  
  In the study of spectral graph theory, the effects on the spectrum are observed when some operations, such as edge moving, edge subdividing, are applied to the graph. In this section, we study the effects on spectral radii of $k$-graphs under the operations:vertex-splitting and vertex-releasing operations.

  \begin{definition}[Edge moving operation \cite{Wang_2020}]
    For $k\ge 2$,  let $G$ be a $k$-uniform hypergraph with $u, v_1, \dots, v_r\in V(G)$ and $e_1, \dots, e_r\in E(G)$ for $r\ge 1$ such that $u \notin e_i$ and $v_i\in e_i$ for $i=1, \dots, r$,  where $v_1, \dots, v_r$ are not necessarily distinct. Let $e'_i=(e_i\backslash \{v_i\})\cup \{u\}$ for $i=1, \dots, r$. Suppose that $e'_i\notin E(G)$ for $i=1, \dots, r$. Let $G'=G-\{e_1, \dots, e_r\}+\{e'_1, \dots, e'_r\}$.
    Then we say that $G'$ is obtained from $G$ by moving edges $(e_1, \ldots, e_r)$ from $ (v_1, \ldots, v_r)$ to $u$.
  \end{definition}

  The following Lemma is a special case of Corollary 2.1 in \cite{Wang_2020}.
  
  \begin{lem}\label{cor1}
    Let $u_1, u_2$ are non-pendant vertices in an edge of connected uniform hypergraph $H$ with $|E_H(u_i)\setminus (E_H(u_1) \cap  E_H(u_2) )|\ge 1$ for $i=1, 2$. Let $H'$ be the hypergraph obtained from $H$ by moving edges $E_H(u_2)\setminus (E_H(u_1) \cap  E_H(u_2) )$ from $u_2$ to $u_1$ and $H\ncong H'$,  then $$\rho(H)<\rho(H').$$
  \end{lem}

  \begin{cor}\label{cor:41}
    Let $H$ be a connected $k$-graph,   $E_H(v)=\{e_1, e_2, \ldots, e_s\}$ ($s\geq 3$),  $u \in e_2$  and $d_H(u)=1$.  Take $H'=H-e_1+e'_1$,  where $e'_1=e_1-\{v\}+\{u\}$.  Then $\alpha(H')>\alpha(H)$.
  \end{cor}
  
  \begin{proof}
   Since $u$  and $v$ are vertices with degree at least 2 in $H'$. Let $H'$ be the hypergraph obtained by moving $e'$ from $u$ to $v$, then $H' \cong H$.  By Lemma \ref{cor1}, we have $\alpha(H')>\alpha(H)$.
  \end{proof}

  \begin{definition}[vertex-splitting operation]
    Let $w$ be a  vertex of a hypergraph $H$ and $H'$ be the hypergraph obtained by replacing $w$ with an edge $e=(u_1, u_2, \dots, u_k)$ in such a way that some vertices adjacent to $w$ are now adjacent to $u_1$ and the rest are adjacent to $u_k$.  $d_{H'}(u_1)+d_{H'}(u_k)=d_H(w)+2$ and $d_{H'}(u_2)=\cdots= d_{H'}(u_{k-1})=1$.
    Then $H'$ is said to be obtained from $H$ by splitting vertex $w$.
  \end{definition}

  \begin{lem} \label{lem:41}
    Let $w$ be a vertex with degree 2 of connected $k$-graph $H$. Suppose $E_H(w)=\{
      e_1, e_2\}$ and $H'$ be the $k$-graph obtained from  $H$ by splitting vertex $w$. Let $\alpha= \alpha(H),  \phi(x)=x^2-x +\alpha$. If $\phi(B_H(w, e_1))<0$,  then\\
    $\alpha(H)< \alpha(H')$ and $\rho(H')<\rho(H)$.
  \end{lem}

  \begin{proof}
    Suppose that  $e'_1, e'_2$ are new edges obtained from $e_1$ and $e_2$ and  $e_0=(u_1, u_2, \dots, u_k)$ be the new edge of $H'$ which obtained from $w$ during the process of splitting vertex $w$. Then $e'_1=(e_1\setminus \{w\})\cup \{u_1\}$ and
    $e'_2=(e_2\setminus \{w\})\cup \{u_k\}$.
    Take $B$ be a weighted incidence matrix of $H'$ such that
    \[
      B(v, e)= \begin{cases}
        B_H(v, e)   & \text{ for $e \not\in \{e_0, e'_1, e'_2\}$ and $v \not\in e_0\cup e'_1\cup e'_2$},     \\
        B_H(v, e)   & \text{ for $e \in \{e'_1, e'_2\}$ and $v \in (e'_1 \cup e'_2)\setminus \{u_1, u_k\}$}, \\
        B_H(w, e_1) & \text{ for  $(v, e)\in \{(u_1, e'_1), (u_k, e_0)\}$},                                  \\
        B_H(w, e_2) & \text{ for  $(v, e)\in \{(u_k, e'_2), (u_1, e_0)\}$},                                  \\
        1           & \text{ for  $(v, e)\in \{(u_i, e_0)\mid 2 \leq i \leq k-1  \}$.}
      \end{cases}
    \]
  
    Then $\forall v \in V(H')  $ and
    $\forall e \in E(H')\setminus \{e_0\}$,  we have $$\sum_{e  \in E_{H'}(v)}B(v, e)=1 \text{  and  } \prod_{v \in e}B(v, e)= \alpha .$$
    Since $B_H(w, e_1)+B_H(w, e_2)=1,  \phi(B_H(w, e_1))<0$,
    $B_H(w, e_1)B_H(w, e_2)>r_1r_2=\alpha$ holds.
    Consequently,
    $\displaystyle  \prod_{v \in e_0}B(v, e_0)= B_H(w, e_1)B_H(w, e_2)> \alpha$.
  
    So $H'$ is  \(\alpha\)-subnormal. By (2) of Theorem \ref{lem_lu},  we have $\alpha(H)< \alpha(H')$ and $\rho(H')<\rho(H)$.
  \end{proof}

  \begin{prop}\label{prop4.3} Let $e=(v_1, v_2, \ldots, v_k)$ be an edge of  connected $k$-graph $H$ with $d(v_1) \geq 2$  and $B$ be the  $\alpha$-normal
    weighted incidence matrix of $H$,  then
    \begin{enumerate}[(1).]
      \item $\alpha \leq B(v_1, e) \leq 1-\alpha$;
      \item if $d(v_1) \geq 2,  d(v_2) \geq 2,  d(v_3) \geq 2$, then $B(v_1, e) \geq \frac{\alpha}{(1-\alpha)^2}$.
    \end{enumerate}
  \end{prop}
  
  \begin{proof}
    (1). Let $e=(v_1, v_2, \ldots, v_k)$ and $v_1 \in e\cap e'$. \\ 
    By the definition of $\alpha$-normal weighted incidence matrix,  we have:\\
    $\prod_{i=1}^{k}B(v_i, e)=\alpha$ and $0<B(v_i, e)\leq 1$. \\ So $\alpha \leq B(v_1, e)$.
    Similarly,  $\alpha \leq B(v_1, e')$.
  
    Since $\displaystyle \sum_{e \in E_H(v)}B(v, e)=1$,  we have  $B(v_1, e)\leq 1-B(v_1, e') \leq 1- \alpha$.
  
    (2).
    From (1),  we have\\
    $\displaystyle B(e, v) = \frac{\alpha}{\prod\limits_{u \in e\setminus \{v\}}B(u, e)} \geq \frac{\alpha}{B(v_2, e)B(v_3, e)} \geq \frac{\alpha}{(1-\alpha)^2}$.
  \end{proof}
  Let
  $$\alpha^{*}=1/3(4-c^{1/3}-c^{-1/3})\approx 0.24512233, $$ where $c=(3\sqrt{69}+25)/2$. Then
  $\alpha^{*}$ is the unique root of the following equations in the interval $(0, 1)$ :
  \begin{align}
     & (1-x)^5=x, \notag                                                        \\
     & (1-x)^2+(1-x)^3=1, \notag                                                \\
     & x-(1-x)^2+(1-x)^4=0,  \label{eq:alpha}                                   \\
     & \left( \frac{1}{2}(1+\sqrt{1-4 x}) \right)^2(1-x)-x=0. \label{eq:alpha1}
  \end{align}
  \begin{claim}\label{claim1}
    For $0<\alpha<\frac{1}{4}$,  take $r_1>r_2$ be the two roots of $\phi(x)=x^2-x+ \alpha=0$.
    We have
    $$\frac{\alpha}{(1-\alpha)^2}>r_2 \iff \alpha <\alpha^{*}.$$
  \end{claim}
  
  \begin{proof}Since $0<\alpha<\frac{1}{4}$,  we have
    $r_1> \frac{1}{2}>\frac{\alpha}{(1-\alpha)^2}$. Then
    $$
      \begin{aligned}
        \frac{\alpha}{(1-\alpha)^2}>r_2 & \iff \phi(\frac{\alpha}{(1-\alpha)^2})<0    \\
                                        & \iff \alpha -(1-\alpha)^2 + (1-\alpha)^4<0.
      \end{aligned}
    $$
    Since $\alpha^*$ is the unique root of  equation \eqref{eq:alpha} in the interval $(0, 1)$,  the claim follows.
  \end{proof}
  
  \begin{lem}\label{claim2}
    Let $\mathbb{C}_n=v_1e_1v_2\cdots v_n e_n v_1$  be a loose cycle with vertex $u \in e_n$ and $d(u)=1$. We use  $\mathbb{C}^{*}_n$,  $\mathbb{C}^{'}_n$ to denote the $k$-graph
    obtained from $\mathbb{C}_n$ by attaching a pendent edge to vertex $u$ and  $v$,  respectively.  We have $$\alpha(\mathbb{C}^{'}_n) < \alpha(\mathbb{C}^{*}_n) < \alpha^*.$$
  \end{lem}
  \begin{proof}

    Let $\alpha_n=\alpha(\mathbb{C}^{*}_n)$,  $B=B_{\mathbb{C}^{*}_n}$ and $y_{i}=B(v_i, e_i)$ for $i=1, \dots,  n$.
  
    Since $\mathbb{C}_n$ is a proper subgraph of $\mathbb{C}^{*}_n$,  we have $\alpha_n < \alpha(\mathbb{C}_n)= \frac{1}{4}$.
  
    Then we have $y_1+y_n=1$ and $(y_i)_{i=1}^{n}$ is a symmetric sequence. According to Theorem \ref{sym_rec},  $\phi(B(v_1, e_1)<0$.
    Let $H$  be $k$-graph obtained from $\mathbb{C}^{*}_n$  by splitting vertex $v_1$. It is clear that $H \cong \mathbb{C}^{*}_{n+1}$. By Lemma  \ref{lem:41},  we have $\alpha(\mathbb{C}^{*}_{n+1})> \alpha(\mathbb{C}^{*}_{n})$. Namely,  $(\alpha_i)_{i=2}^{+\infty}$ is a strictly increasing  bounded  sequence. So 	$\displaystyle \lim_{i\to +\infty}\alpha_i$ exists. Suppose $\displaystyle \lim_{i\to +\infty}\alpha_i=\alpha$. Let $\displaystyle \theta_n=\arcosh(\frac{1}{2} \alpha_n^{-\frac{1}{2}}),  \theta= \lim_{n\to \infty}\theta_n$.
    Because $B$ is the $\alpha_n$-normal weighted incidence matrix of $\mathbb{C}_n^*$,   from Theorem \ref{sym_rec}  it follows that $\theta_{n}$ is the root of the following equation:
    $$\frac{1}{4}\bigg(1+ \tanh(\theta_{n})\tanh\left(\frac{(n-1)\theta_n}{2}\right)\bigg)^2(1-\alpha_n)=\alpha_n.$$
  
    We take limits as $i\to \infty$ on both sides of above equation. We get
    $$\frac{1}{4}\bigg(1+ \tanh(\theta)\bigg)^2(1-\alpha)=\alpha .$$
    That  is $r_1^2(1-\alpha)=\alpha$,  where $r_1$ is the largest root of
    $x^2-x+\alpha=0$. From equation \eqref{eq:alpha1},  we have $\alpha=\alpha^*$.
    So $\alpha_n<\alpha^*$. From Lemma  \ref{lem:41},  $\alpha(\mathbb{C}^{'}_n) < \alpha(\mathbb{C}^{*}_n)$. So $$\alpha(\mathbb{C}^{'}_n) < \alpha(\mathbb{C}^{*}_n) < \alpha^*.$$
    This result follows.\end{proof}
  \begin{thm}\label{thm41}
    Suppose $n \geq 0$.
    Let $e', e''$ be two edges of connected $k$-graph $H$ and $w_0 \in e',  w_n \in e'',  d_H(w_0)=d_H(w_n)=2$,  $e', e''$ both contains at least three vertices with degrees exceeding 1.  $\mathbb{P}_{n}$ is an internal path of $H$ between $w_0$ and $w_n$.
    Let $H'$ be a $k$-graph obtained from $H$ by splitting vertex $w_0$. \\
    If $\alpha(H)< \alpha^*$,   then $\rho(H)<\rho(H')$.
  \end{thm}
  \begin{proof}
    Suppose $\mathbb{P}_{n}=w_0e_1w_1e_2w_2 \cdots e_{n-1}w_{n-1}e_nw_{n}$.
  
    Write $x_i=B(w_i, e_{i+1}),  x'_{i}=B(w_i, e_{i})$ for $i\in \{1, \ldots, n-1\}$ and \\ $x_0=B(w_0, e_1),  x_n=B(w_n, e''),  x'_0=B(w_0, e'),  x'_n=B(w_n, e_n)$.
    Take $\alpha=\alpha(H)$. From Proposition \ref{prop4.3},  $\min(B(w_0, e'), B(w_n, e''))\geq \frac{\alpha}{(1-\alpha)^2}$.
  
    Since $\alpha< \alpha^*$,  by Claim \ref{claim1},  we have \\
    $x_0=B(w_0, e_1)=1-B(w_0, e') \leq 1- \frac{\alpha}{(1-\alpha)^2}< 1- r_2 = r_1$. \\
    Similarly,  $x'_n=B(w_n,  e_n)<r_1$.
  
    According to Lemma  \ref{lem:41},  to complete the proof,  it suffices to show  $\phi(x_0)<0$. That is,  $r_1> x_0 >r_2$.
  
    On the contrary,  suppose $x_0\leq r_2$. By Lemma \ref{thm-recur},  $(x_i)_{i=0}^{n}$ is decreasing. So $x_{n-1} \leq x_0 \leq r_2$.
    Since $x_{n-1}x'_n=\alpha$,  $x'_n=\frac{\alpha}{x_{n-1}} \geq \frac{\alpha}{r_2}=r_1$,   a contradiction. Hence,  $r_2<x_0<r_1$,  $\phi(x_0)<0$.  The result follows from Lemma  \ref{lem:41}.
  \end{proof}

  \begin{definition}[Vertex-releasing operation]
    Let $e$ be an edge of connected $k$-graph $H$ and $v_1,v_2,\ldots,v_s$ ($s \geq 2$) be all vertices with degrees at least 2 in $e$.
    Take $e'=(e\cup \{w_1,\ldots,w_t\}) \setminus \{v_1,\ldots,v_t\}$ ($1 \leq t < s$), where $w_1,\ldots, w_t$ are new vertices. Let $H'$ be  hypergraph such that $E(H')=(E(H)\setminus \{e\})\cup \{e'\}$.  We say that $H'$ is obtained from $H$ by releasing vertex $v_1,\ldots,v_t$ of $e$.
  \end{definition}

  \begin{thm}\label{thm:vertex_release}
    Let $e$ be an edge of connected $k$-graph $H$ and $v_1,v_2,\ldots,v_s$ ($s \geq 2$) be all vertices with degrees at least 2 in $e$. Take $H'$ be $k$-graph obtained from $H$ by releasing some vertices of $e$.  We have
    $$\alpha(H')> \alpha(H) \text{ and }  \rho(H')< \rho(H).$$
  \end{thm}
  
  \begin{proof}
    Let $\alpha=\alpha(H)$ and $e'=(e\cup \{w_1,\ldots,w_t\}) \setminus \{v_1,\ldots,v_t\}$.\\
    Suppose $V(H')=V(H)\cup \{w_1,\ldots,w_t\}$ and $E(H')=(E(H)\setminus \{e\})\cup \{e'\}$.
    Let $B_H$ be the consistently $\alpha$-normal weighted incidence matrix of $H$.
  
    Take $B$ be a weighted incidence matrix of $H'$ such that
    \[
      B(u,f)=
      \begin{cases}
        B_H(u,f) & \text{ for $u \not\in \{w_1,\ldots,w_t\}$ and $f \neq e'$}, \\
        B_H(u,e) & \text{ for $u \not\in \{w_1,\ldots,w_t\}$ and $f=e'$},      \\
        B_H(v_i,e) & \text{ for $f=e'$ and $u=w_i$, $i \in \{1,\ldots,s\}$ }.
      \end{cases} \vspace*{-2pt}
    \]
    Then,  for any $f$ in $E(H')$, $\displaystyle \prod_{u \in f} B(u,f)=\alpha$. \\
    For any $u$ in $V(H')\setminus \{v,w\}$, $\displaystyle \sum_{f \in E_{H'}(u)} B(u,f)=1$ \\and
    for $i \in \{1,\ldots,s\}$, \\
    $\displaystyle \sum_{f \in E_{H'}(v_i)} B(v_i,f) = 1- B_H(v_i,e)<1$   and $\displaystyle \sum_{f \in E_{H'}(w_i)} B(w_i,f)= B(v_i,e)<1$.
  
    Hence, $B$ is strictly $\alpha$-subnormal weighted incidence matrix of $H'$. \\ The result follows from Theorem \ref{lem_lu}.
  \end{proof}

  \begin{thm}\label{thm:43}
    Let $H$ be a connected $k$-graph with $c(H)>0$. If $H$ is not loose cycle,  then $\alpha(H)<\alpha^*$.
  \end{thm}
  \begin{proof}
    Let $G$ be a minimal $c$-cyclic subgraph of $H$ with $c>0$.
  
    \noindent \textbf{Case 1}. $G$ is a loose cycle.
  
    When there exist a subgraph of  $H$ isomorphic to one of the following two graphs:$\mathbb{C}_n^*$  and $\mathbb{C}_n^{'}$, where $n$ is some integer more than one,   by Lemma \ref{claim2}, $\alpha(G) < \alpha^*$. Otherwise,  since $H$ is connected and is not loose cycle, there must exist $e \in E(H)$ such that $|e\cap V(G)| \geq 2$. Let $H_1$ be hypergraph with $V(H_1)=V(H)\cup e$ and $E(H_1)=E(G)\cup \{e\}$. A hypergraph $G_1$ can be obtained by application of vertex-releasing operation to vertices in $e$ of degree more than one of  $H_1$ such that $G_1$ is isomorphic to $\mathbb{C}_n^*$ or $\mathbb{C}_n^{'}$. By Theorems   \ref{thm:vertex_release}, \ref{lem_lu}, we have $\alpha(H)\leq \alpha(G_1) < \alpha(C^*_0) <\alpha$.
  
    \noindent \textbf{Case 2}.
    $G$ is not loose cycle.  \\
    Since $c(H)>1$, $G$ must be hypergraph with $E(G)=\{e_1,e_2\}$ and $|e_1 \cap e_2| \geq 3$.
  
    Write $p=|e_1 \cap e_2|$.  Then we have $\alpha(H)\leq \alpha(G)= \frac{1}{2^p} < \alpha^* $.
  \end{proof}
  
  By Theorems
  \ref{thm41} and \ref{thm:43}, the following corollary is immediate.
  \begin{cor}\label{cor:42}
    Let $H$ be a connected $k$-graph with $c(H)>0$ and $v$ is a vertex with degree 2 of some internal path of $H$. Let $H'$ be a $k$-graph obtained from $H$ by splitting vertex $w_0$. Then $\rho(H)<\rho(H')$.
  \end{cor}

  \section{Extremal hypergraphs in \texorpdfstring{$\mathscr{B}(m)$}{ B(m)}  with the smallest spectral radius}

  \begin{figure}
    \centering
    \begin{adjustbox}{width=\textwidth}
      \subfloat[][$\infty(n_1, n_2, n_3)$]{
        \includegraphics[page=15]{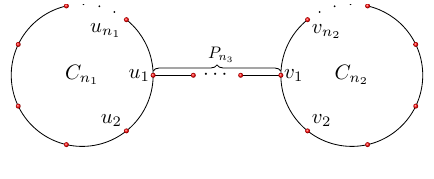}
      }
      \subfloat[][$\infty(n_1, n_2)$]{
        \includegraphics[page=16]{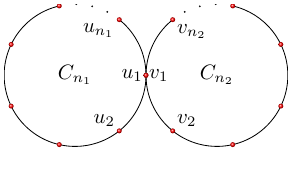}
      }
    \end{adjustbox} \\
    \subfloat[][$\theta(n_1, n_2, n_3)$]{
      \includegraphics[page=17]{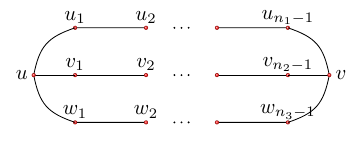}
    }
    \caption{$\infty$ graph and $\theta$ graph}
    \label{bicyclicfig}
  \end{figure}

  Let $H$ be a $c$-cyclic hypergraph. The base of $H$,  denoted by $\widehat{H}$,  is the (unique)  minimal $c$-cyclic subgraph of $H$.  It is easy to see that $\widehat{H}$
  can be obtained from $H$ by consecutively deleting pendent edges.


  It is known that there are three types of bases of bicyclic graphs (see Fig. \ref{bicyclicfig}). $\infty(n_1, n_2)$ is obtained from two vertex-disjoint cycles $C_{n_1}$ and $C_{n_2}$ by identifying vertex $u$ of $C_{n_1}$ and vertex $v$ of $C_{n_2}$,   $\infty(n_1, n_2, n_3)$ is obtained from two vertex-disjoint cycles  $C_{n_1}$ and $C_{n_2}$ by joining vertex $u$ of $C_{n_1}$ and vertex $v$ of $C_{n_2}$ by a path $uu_1u_2\cdots u_{n_3-1}v $ of length $n_3(n_3 \geq 1)$,  $\theta(n_1, n_2, n_3)$  ($n_1\geq n_2 \geq n_3$) is obtained from three pairwise internal disjoint paths of lengths $n_1, n_2, n_3$ from vertices $u$ to $v$. Note that $\infty(n_1, n_2, 0)$ is exactly $\infty(n_1, n_2)$.
  
  The bicyclic graph containing $\theta(p, l, q)$ as its base is called a $\theta$-graph.
  
  For a hypergraph $H =(V,  E)$,  we define the incidence graph to be the bipartite graph $K(H)$ with the vertex set $V \cup E$ and the edge set
  $$
    \{(v,  e):(v,  e) \in V \times E,  v \in e\}.
  $$
  
  The graph $K(H)$ is also called the \emph{K\"onig representation } of hypergraph $H$.

  It is easy to see that the number of Berge  cycles of hypergraph $H$ and the number of cycles of its K\"onig representation $K(H)$ is same.

  \begin{figure}
    \centering
    \begin{tikzpicture}
      \begin{scope}
        \node {\includegraphics[page=1, scale=1]{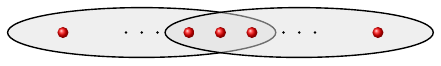}};
        \node at (0,  -1.6){(a) $C_0$};
      \end{scope}
      \begin{scope}[xshift=0.5 \textwidth,  yshift=3mm]
        \node {\includegraphics[page=11, scale=0.8]{figures/figure}};
        \node at (0,  -2){(b) $\widehat{K(C_0)}$};
      \end{scope}
    \end{tikzpicture}
    \caption{The bicyclic hypergraph $C_0$ and $\widehat{K(C_0)}$}
    \label{c3tg}
  \end{figure}
  
  \begin{prop}
    Let $H$ be a connected $k$-uniform hypergraph,  then $c(H)=c(K(H))$.
  \end{prop}
  
  \begin{proof}Suppose $H=(V, E)$ be a connected
    $k$-uniform hypergraph with order $n$ and size $m$. Since
    By the definition of $K(H)$,  we have $|V(K(H))|=m+n$ and $|E(K(H))|=km$. So $c(K(H))=|E(K(H))|-|V(K(H))|+1=(k-1)m-n+1=c(H)$.
  \end{proof}
  
  Inspired by the above result,  the definition of the cyclomatic number for $k$-graph can be generalized to general hypergraph as follows:
  The cyclomatic number of general hypergraph $H$,  denote  by $c(H)$,  is defined as the cyclomatic number of the K\"onig representation of $H$,  namely $c(H)=c(K(H))$.
  
  Let $H$ be a $c$-cyclic hypergraph and $H'$ be hypergraph obtained by adding some new vertices with degree 1 to some edges of $H$.  It is obvious that $c(H)=c(H')$.

  Let $H$ be a $k$-graph and $e =\{u, u_1, \ldots, u_{k-2}, v \}$ be an edge of $H$,  where $d(u_1) = \cdots = d(u_{k-2}) = 1$.
  Subdividing an edge $e$ means replacing edge $e$ by two edges $f_1 = \{u, u'_1, \ldots,   u'_{k-2}, w \}$ and $f_2 = \{w,  u_1, \ldots,   u_{k-2}, v \}$,  where
  $w,  u'_1, \ldots,   u'_{k-2}$ are new added vertices.

  Let $\mathscr{B}(m)$ denote the set consisting of all
  bicyclic $k$-graphs with size $m$.
  In the remaining part of this section,  we will characterize the extremal hypergraph with the smallest spectral radius  among bicyclic $k$-graph with given size.

  Let $\mathscr{C}(m)$  consist of all the minimal bicyclic $k$-graphs ($k \geq 3$) with edge number at most $m$.
  Let $\mathscr{C}_1(m)=\{G\mid G\in \mathscr{C}(m),  \Delta(G)=2\}$,
  $\mathscr{C}_2(m)=\{G\mid G\in \mathscr{C}(m), \Delta(G) \geq  3\}$.
  
  \begin{lem}\label{lem:51}
    Let $H$ be hypergraph in $\mathscr{B}(m)$ with the smallest spectral radius. Then $H$ is hypergraph in $\mathscr{C}_1(m)$.
  \end{lem}
  \begin{proof}
    Let $G$ be hypergraph in $\mathscr{C}(m)$ with the smallest spectral radius. We will show $G \in \mathscr{C}_1(m)$ with $|E(G)|=m$.
  
  
   Assume $G \in \mathscr{C}_2(m)$. 
    Since the K\"onig representation $K(G)$ is bicyclic graph,   $\widehat{K(G)}$ is one of  $\infty(n_1, n_2)$, $\infty(n_1, n_2, n_3)$ and $\theta(n_1, n_2, n_3)$ and  there exists vertex $v \in \widehat{K(G)}$ with $d_{\widehat{K(G)}}(v)\geq 3$ which is a vertex $u$ of $G$  with $d_G(u) \geq 3$.
      Since $k \geq 3$,  there exist  $e \in  E(G)$ and  $u,  v \in  e$  such that $d_G(u)=1$ and $d_G(v)\geq  3$. From Corollary  \ref{cor:41},  there exists $k$-graph ($k \geq 3$) $G' \in \mathscr{C}(m)$ with $|E(G')|=m$  such  that $\alpha(G')> \alpha(G)$. Namely,  $\rho(G')<\rho(G)$,   which is a contradiction to our assumption on $G$.  
  So $G \in\mathscr{C}_1(m)$.   If $|E(G)|<m$,  let $G'$ be a hypergraph obtained by splitting some vertex $v$ with $d(v)=2$. By Corollary \ref{cor:42},  $\rho(G')< \rho(G)$ and $|E(G')|\leq m$,  a contradiction. So $|E(G)|=m$ and $G \in \mathscr{B}(m) \cap \mathscr{C}_1(m)$.
  
  Assume $H \not\in \mathscr{C}(m)$. Then we have $\widehat{H}$ is proper subgraph of $H$ and $\widehat{H} \in \mathscr{C}(m)$. By Theorem \ref{proper:subgraph} we have $\rho(G)\leq \rho(\widehat{H})<\rho(H)$, which is a contradiction to our assumption on $H$ since $G \in  \mathscr{B}(m)$.
   Therefore, $H$ is hypergraph in $\mathscr{C}_1(m)$ with size $m$.
  \end{proof}

  \begin{figure}
    \centering
    \begin{adjustbox}{width=\textwidth}
      \subfloat[][$C_1(n_1, n_2,  n_3)$]{\label{figC1}
        \includegraphics[scale=0.8, page=9]{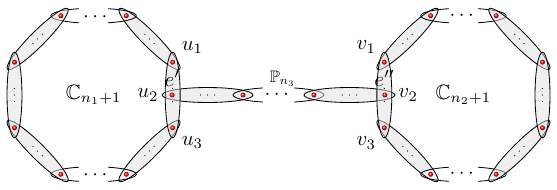}
      }
      \subfloat[][$C_2(n_1, n_2,  n_3)$]{\label{figC2}
        \includegraphics[scale=1.2,  page=4, width=0.5\textwidth]{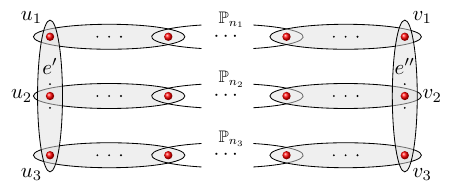}
      }
    \end{adjustbox}
    \caption{$\infty$-type and $\Theta$ type hypergraph}
    \label{figC1C2}
  \end{figure}
  
  Let $e',  e''$ be hyperedges with  $|e'|=|e''|=k\geq 3$ and $e'\cap e''=\emptyset$.
  
  Take $U=\{u_1, u_2, u_3\}, V=\{v_1, v_2, v_3\}$ such that $U \subseteq e'$,  $V \subseteq e''$.
  Let $n_1,  n_2,  n_3$  be non-negative integers. We denote  by $C_1(n_{1},  n_{2},  n_{3})$ the $k$-graph (See (a) of Fig. \ref{figC1C2}) obtained  from $e'$ and $e''$ by joining vertex pairs $(u_1,  u_3)$,  $(v_1, v_3)$ and $(u_2, v_2)$ with three internal loose paths of lengths $n_1,  n_2,  n_3$,  respectively. $C_1(n_{1},  n_{2},  n_{3})$ is also referred to as  $\infty$-type hypergraph.
  
  Let $C_2(n_{1},  n_{2},  n_{3})$  be the $k$-graph  (See (b) of Fig. \ref{figC1C2}) 
  obtained  from $e'$ and $e''$ by joining vertex pairs $(u_1, v_1)$,  $(u_2, v_2)$,  $(u_3, v_3)$ with three  internal loose paths of lengths $n_1, n_2, n_3$,  respectively. $C_2(n_{1},  n_{2},  n_{3})$ is also referred to as  $\Theta$-type hypergraph.

  Let $u_1, u_2, v_1, v_2$ be distinct vertices in hyperedge $e$. 
  We denote  by $C_3(n_{1},  n_{2})$ the $k$-graph obtained from $e$ by joining vertex pairs $(u_1, u_2)$,  $(v_1, v_2)$ with two  internal loose paths of lengths $n_1, n_2$,  respectively. $C_3(n_{1},  n_{2})$ is also referred to as  $\Theta^*$-type hypergraph.
  
  Suppose $H \in \mathscr{C}_1(m)$.  Those  vertices with degree more than 3 in $\widehat{K(H)}$ must be edges  of $H$. It is clear that
  \begin{enumerate}[(1).]
    \item  when $\widehat{K(H)}$ is isomorphic to some $\infty(p, q, r)$($r\geq 1$),   $H$ is $\infty$-type $k$-graph,
    \item when $\widehat{K(H)}$ is isomorphic to some $\infty(p, q)$,   $H$ is $\Theta^*$-type $k$-graph,
    \item when $\widehat{K(H)}$ is isomorphic to some $\theta(p, q, r)$,   $H$ is $\Theta$-type $k$-graph.
  \end{enumerate}

  The bicyclic hypergraphs in $\mathscr{C}_1(m)$ can be partitioned into the following three sets:
  $$
    \begin{aligned}
      \mathscr{B}_1(m)= & \{G\mid G \in \mathscr{B}(m) \text{ and } \widehat{G} \text{ is $\infty$-type $k$-graph}\},   \\
      \mathscr{B}_2(m)= & \{G\mid G \in \mathscr{B}(m) \text{ and } \widehat{G} \text{ is $\Theta$-type $k$-graph}\},   \\
      \mathscr{B}_3(m)= & \{G\mid G \in \mathscr{B}(m) \text{ and } \widehat{G} \text{ is $\Theta^*$-type $k$-graph}\}.
    \end{aligned}
  $$

  \begin{lem}\label{lem:52}
    For  any positive integers $p, q$,  we have $$\rho(C_1(p, p, q))=\rho(C_2(p, p, q)).$$
  \end{lem}
  
  \begin{proof}
    Let $G=C_1(p, p, q),  H=C_2(p, p, q)$.
    Take $\alpha=\alpha(G)$. \\
     Suppose $B_G$ is  the consistently $\alpha$-normal weighted incidence matrix of $G$, then  we have  $B_G(u_1, e')=B_G(u_3, e')=B_G(v_1, e'')= B_G(v_3, e'')$.  \\
    Let $E_{G}(u_1)=\{e', e_1\},  E_{G}(u_3)=\{e', e_{p}\}$,  $E_{G}(v_1)=\{e'', f_1\},  E_{G}(v_3)=\{e'', f_q\}$. Take $\tilde{e}_1=(e_1\cup \{v_3\})\setminus \{u_1\}$,  $\tilde{f}_1=(f_1\cup \{u_3\})\setminus \{v_1\}$.
    Then  $H \cong G\cup \{\tilde{e}_1,  \tilde{f}_1\}\setminus \{e_1,  f_1\}$.
    By exchanging the values of $B_G(u_1,  e_1)$,  $B_G(v_3, e_1)$ and $B_G(v_1, f_1)$,  $B_G(u_3, f_1)$,  respectively,   a consistently $\alpha$-normal weighted incidence  matrix of $G$ is obtained from $B_G$.  Therefore,  $\alpha(G)=  \alpha(H)$ and $\rho(H)=\rho(G)$. \end{proof}

  \begin{thm}\label{thm:51}
    Let $p, q$ be positive integers such that $2p+q=m-2,  |p-q| \leq 1$. Then $C_1(p, p, q), C_2(p, p, q)$ are the extremal hypergraph with the smallest spectral radius in
    $\mathscr{B}_1(m), \mathscr{B}_2(m)$,  respectively.
  \end{thm}
  
  \begin{proof}
    Suppose $H_i$ is the extremal hypergraph with the smallest spectral radius in $\mathscr{B}_i(m)$. Then $H_i=(\widehat{H}_i)$ ($i=1, 2$).  Without loss of generality,  Let  $H_i=C_i(l_1, l_2, l_3)$.  Take $H_i^*=C_i(p, p, q)$,
    $\alpha= \alpha(H_i^*)$.  Let $(y_j)$ be the iterated sequence of $f(x)=1- \frac{\alpha}{x}$ with $y_0=F_0(q),  y_q=F_0^*(q)$. By the symmetry of $H_i^*$,  $y_0+y_p=1$ holds. So  $(y_j)$ is a symmetric $\alpha$-normal sequence.  By
    Corollary \ref{cor3.1},  we have $y_n=F_0^*(2n-q)$.
    If $H_i \neq H_i^*$,  then $(l_1, l_2, l_3)\neq (p, p, q)$.
    Since $\alpha< \frac{1}{4}$,   By Corollary  \ref{cor3.1},  there exists a weighted incidence matrix $B$ of $H_i$ which satisfies the following conditions.
    \begin{enumerate}[(1).]
      \item When $i=1$,  for the internal loose paths  $(u_1, u_3)$-$P_{l_1}$,   $(v_1, v_3)$-$P_{l_2}$ and $(u_2, v_2)$-$P_{l_3}$ of $H_1$,  $B$ is $(F_0(l_1), F_0(l_1))$  $\alpha$-normal,
            $(F_0(l_2), F_0(l_2))$ $\alpha$-normal and $(F_0(m-2-2l_1), F_0(m-2-2l_2))$ $\alpha$-normal.
            And
            \[
              \begin{aligned}
                B(e', u_1)  & =B(e', u_3)=F_0^*(l_1),   & B(e', u_2) & = F_0^*(m-2-2l_1), \\
                B(e'', v_1) & =B(e'', v_3)=F_0^*(l_2),  &
                B(e'', v_2) & =F_0^*(m-2-2l_2).                                           \\
              \end{aligned}
            \]
  
            Without loss of generality,  suppose
            $l_1 \leq l_2$. Since $l_1+l_2+l_3=m-2$,  $m-2-2l_1=l_2+l_3-l_1 \geq 0$.
  
            Since $F_0^*(x)$ is log-concave on $[-1, +\infty)$,
            From
            $2p+q=l_1+l_2+l_3$,   $(l_1, l_2, l_3) \neq (p, p, q)$,  we have
            $$\prod\limits_{v\in e'}B(v, e')=F_0^*(l_1)^2F_0^*(m-2-2l_1)<F_0^*(p)^2F_0^*(q)=\alpha.$$
  
            $F_0^*(l_2)^2F_0^*((m-2-2l_2))<F_0^*(p)^2F_0^*(q)=\alpha$ when $m-2-2l_2 \geq 0$ and
  
            $F_0^*(l_2)F_0^*(0)F_0^*(m-2-l_2)<F_0^*(p)^2F_0^*(q)=\alpha$ when $m-2-2l_2 < 0$.
  
            By (2) of Proposition \ref{prop32},
            $F_0^*(l_2)^2F_0^*((m-2-2l_2))< F_0^*(l_2)F_0^*(0)F_0^*(m-2-l_2)$.  Hence,
            $\prod\limits_{v\in e''}B(v, e'')=F_0^*(l_2)^2F_0^*((m-2-2l_2))< \alpha$.
  
      \item When $i=2$,  \\
            for each $j \in \{1, 2, 3\}$,  the internal loose path  $(u_j, v_j)$-$P_{l_j}$ of $H_2$,  $B$ are $(F_0(l_1), F_0(l_1))$ $\alpha$-normal
            and  $B(e', u_j)  =B(e'', v_j)=F_0^*(l_j)$.
  
            Since $F_0^*(x)$ is log-concave in $[-1, +\infty)$,  \\
            $\prod\limits_{v\in e'}B(v, e')=\prod\limits_{v\in e''}B(v, e'')=F_0^*(l_1) F_0^*(l_2) F_0^*(l_3)< F_0^*(p)^2F_0^*(q)=\alpha$.
    \end{enumerate}
  
    Therefore,  for $i=1, 2$,  we have\\
    $\forall v \in V(H_i)$,  $\displaystyle \sum\limits_{e \in E(v)}B(v, e) = 1;$
    $ \forall e \in E(H)-\{e', e''\}$,
    $\displaystyle \prod\limits_{v\in e}B(v, e) = \alpha$.

    By (1) of Lemma \ref{prop-consistent},  $B_i$ is consistent  with all cycles of $H_i$ for $i=1, 2$.
  
    That is,  $H_i$ is consistently $\alpha$-supernormal,
    by Lemma \ref{prop-consistent},
    $\rho(H_i)>\rho(H^{*}_i)$,   which is a contradiction to our assumption on $H_i$.
  
    So $H_i=C_i(p, p,  q)$ is  the extremal hypergraph with the smallest spectral radius in
    $\mathscr{B}_i(m)$  for $i=1, 2$.\end{proof}

  \begin{thm}\label{thm:52}
    Let $p, q$ be positive integers such that $p+q=m-1$,  $|p-q| \leq 1$. Then $C_3(p,  q)$ is the hypergraph with the smallest spectral radius in
    $\mathscr{B}_3(m)$.
  \end{thm}

  \begin{proof}
    Let $H$ be hypergraph in $\mathscr{B}_3(m)$ with the smallest spectral radius. By similar arguments of proof of Lemma   \ref{lem:51},   $H=\widehat{H}$. That is,  $H$ is a $\Theta^*$ type hypergraph. Without loss of generality,    suppose $H=C_3(l_1, l_2)$ and $l_1+l_2=m-1$. Let $H^*=C_3(p, q)$,   $\alpha=\alpha(H^*)$.  Suppose that $(y_j)$ is the iteration sequence of $f(x)=1- \frac{\alpha}{x}$ with $y_0=F_0(q),  y_q=F_0^*(q)$. By the symmetry of $H^*$,  we have $y_0+y_p=1$ and    $(y_j)$ is a $\alpha$-normal sequence.  From
    Corollary \ref{cor3.1},  $y_n=F_0^*(2n-q)$.
  
    Since $\alpha< \frac{1}{4}$,   according to Corollary \ref{cor3.1},  we can build a weighted incidence matrix $B$ of
    $H$ such  that
    $B$ is $(F_0(l_1),  F_0(l_1))$ $\alpha$-normal for the loose path $(u_1, u_2)$-$P_{l_1}$ and $(F_0(l_2), F_0(l_2))$ $\alpha$-normal for $(v_1, v_2)$-$P_{l_2}$,  \\
    and      $\begin{aligned}
        B(e', u_1) & =B(e', u_2)=F_0^*(l_1), &
        B(e', v_1) & =B(e', v_2)=F_0^*(l_2).   \\
      \end{aligned}$
  
    Then
    $\forall v \in V(H)$,  $\displaystyle \sum\limits_{e \in E(v)}B(v, e) = 1$ and
    $ \forall e \in E(H)-\{e'\}$,
    $\displaystyle \prod\limits_{v\in e}B(v, e) = \alpha$.

    For $e'$,  $\prod\limits_{v\in e'}B(v, e')=F_0^*(l_1)^2F_0^*(l_2)^2<F_0^*(p)^2F_0^*(q)^2=\alpha$.
  
    By (1) of Lemma \ref{prop-consistent},  $B_i$ is consistent for any cycle of $H_i$ for $i=1, 2$.

    Therefore,   $H$ is consistently $\alpha$-supernormal. According to (3) of  Theorem  \ref{lem_lu},
    $\rho(H)>\rho(H^{*})$,   which is a contradiction to our assumption on $H$.
  
    Hence,   $H=C_3(p, q)$ is the hypergraph with the smallest spectral radius in   $\mathscr{B}_{3}(m)$.
  \end{proof}
  
  \begin{thm}\label{thm:53}
    Let integers $p, q$ such that $p\geq q,  p+q=m-1,  |p-q|\leq 1$. Take hypergraphs $H= C_2(p-1, q-1, 1)$,  $G=  C_3(p, q)$. Then $\rho(H)<\rho(G)$.
  \end{thm}
  \begin{proof}
    Set $\alpha=\alpha(G)$. Let $(y_j)$ be the iteration sequence of $f(x)=1- \frac{\alpha}{x}$ with $y_0=F_0(q),  y_q=F_0^*(q)$. By the symmetry of $G$,  $y_0+y_p=1$.  So $(y_j)$ is a symmetric $\alpha$-normal sequence.  By Corollary   \ref{cor3.1},  we have $y_n=F_0^*(2n-q)$.
    Since $\alpha< \frac{1}{4}$,  by Corollary  \ref{cor3.1},  there exists a weighted incidence matrix $B$ of $H$ such  that
    \\
    $B$ is $(F_0(l_1), F_0(l_1))$ $\alpha$-normal for $(u_1, u_2)$-$P_{l_1}$ and $(F_0(l_2), F_0(l_2))$ $\alpha$-normal for $(v_1, v_2)$-$P_{l_2}$.
    And,
    \[
      \begin{aligned}
        B(e', u_1) & =B(e', v_1)=F_0^*(p-1), \\
        B(e', u_2) & =B(e', v_2)=F_0^*(q-1), \\
        B(e', u_3) & = B(e'', v_3)=F_0^*(1), \\
      \end{aligned}
    \]
  
    $\forall v \in V(H)$,  $\displaystyle \sum\limits_{e \in E(v)}B(v, e) = 1, $ \quad
    $ \forall e \in E(H)-\{e', e''\}$,
    $\displaystyle \prod\limits_{v\in e}B(v, e) = \alpha$.
  
    Since $F_0^*(p)^2F_0^*(q)^2=\alpha$,  $F_0^*(x)>0$,  $F_0^*(p)F_0^*(q)=\sqrt{\alpha}=F_0^*(-1)$.
  
    For $e', e''$,    $$
      \begin{aligned}
        \prod\limits_{v\in e'}B(v, e')=\prod\limits_{v\in e''}B(v, e'')= & F_0^*(p-1)F_0^*(q-1)F_0^*(1)      \\
        >                                                                & F_0^*(p)F_0^*(q)F_0^*(-1)=\alpha.
      \end{aligned}
    $$
  
    By (2) of Theorem \ref{lem_lu},  $H$ is $\alpha$-subnormal. $\rho(H)<\rho(G)$ holds.
  \end{proof}

  From Lemmas  \ref{lem:51}, \ref{lem:52} and Theorems \ref{thm:51},   \ref{thm:52},  \ref{thm:53},  we derive our main result.
  
  \begin{thm}\label{mainresult}
    Let $H_1$ and $H_2$ be the hypergraph with the smallest spectral radius in $\mathscr{B}_{1}(m), \mathscr{B}_{2}(m)$,  respectively. Then
    \begin{enumerate}[(1).]
      \item $\rho(H_1)=\rho(H_2)$,
      \item $H_1, H_2$ are the hypergraphs with the smallest spectral radius in $\mathscr{B}(m)$.
    \end{enumerate}
  \end{thm}
  
  According to those results obtained, we can compute the values of spectral radii of $k$-graphs with the smallest spectral radius.
  
  Let $\rho$ be the smallest spectral radius of bicyclic $k$-graph with size $m$ and $q=\lfloor \frac{m-2}{3} \rfloor$. Suppose
  $$F(\theta)=\begin{cases}
      F_0^*(q)^2F_0^*(q+1) & \text{for  $m\equiv 0 \pmod{3}$ }, \\
      F_0^*(q)F_0^*(q+1)^2 & \text{for $m\equiv 1 \pmod{3}$ },  \\
      F_0^*(q)^3           & \text{for  $m\equiv 2 \pmod{3}$ }. \\
    \end{cases}
  $$
  Let $\theta_0$ be the unique positive root of  equation \[
    F(\theta)-\frac{1}{4} \sech(\theta)^2=0.
  \]
  Then $\rho= (2\cosh(\theta_0))^{\frac{2}{k}}$.
  
  As a example, we take $H$  to be a $k$-graph with the smallest spectral radius in $\mathscr{B}(14)$ for $k=3$. Let $\alpha=\alpha(H)$ and $\theta_0= \arcosh(\frac{1}{2} \alpha^{-\frac{1}{2}})$. Then $\theta_0$ is the positive root of the following equation
  \[
    (1 + \tanh(\theta)\tanh(2\theta))^3- 2 \sech^2(\theta)=0.
  \]
  
  By direct calculation, we have  $\theta_0 \approx 0.35581$ and $\alpha(H) \approx 0.22084$.  Hence, $\rho(H)=(2\cosh(\theta_0))^{\frac{2}{3}} \approx 1.654396$.
\addcontentsline{toc}{section}{References}

\begin{thebibliography}{10}
  
  \bibitem{Berge1973}
  C.~Berge, Graphs and hypergraphs, North-Holland Publishing Co., New York, 1973.
  
  \bibitem{Fan2016a}
  Y.-Z. Fan, Y.-Y. Tan, X.-X. Peng, A.-H. Liu, Maximizing spectral radii of
    uniform hypergraphs with few edges, Discuss. Math. Graph Theory 36~(4) (2016)
    845--856.
  
  \bibitem{Li2016a}
  H.~Li, J.-Y. Shao, L.~Qi, The extremal spectral radii of $k$-uniform
    supertrees, Journal of Combinatorial Optimization 32~(3) (2016) 741--764.
  
  \bibitem{xiao2017maximum}
  P.~Xiao, L.~Wang, Y.~Lu, The maximum spectral radii of uniform supertrees with
    given degree sequences, Linear Algebra and its Applications 523 (2017)
    33--45.
  
  \bibitem{Yuan2016a}
  X.~Yuan, J.~Shao, H.~Shan, Ordering of some uniform supertrees with larger
    spectral radii, Linear Algebra and Its Applications 495 (2016) 206--222.
  
  \bibitem{lu2016connected}
  L.~Lu, S.~Man, Connected hypergraphs with small spectral radius, Linear Algebra
    and Its Applications 509 (2016) 206--227.
  
  \bibitem{Ouyang2017}
  C.~Ouyang, L.~Qi, X.~Yuan, The first few unicyclic and bicyclic hypergraphs
    with largest spectral radii, Linear Algebra Appl. 527 (2017) 141--162.
  
  \bibitem{Zhang_2020}
  J.~Zhang, J.~Li, H.~Guo,
    {Uniform hypergraphs with
    the first two smallest spectral radii}, Linear Algebra and its Applications
    594 (2020) 71--80.
  
  \bibitem{Lim2005}
  L.-H. Lim, Singular values and eigenvalues of tensors: a variational approach,
    in: Computational Advances in Multi-Sensor Adaptive Processing, 2005 1st IEEE
    International Workshop on, IEEE, 2005, pp. 129--132.
  
  \bibitem{Qi2005}
  L.~Qi, Eigenvalues of a real supersymmetric tensor, J. Symbolic Comput. 40~(6)
    (2005) 1302--1324.
  
  \bibitem{cooper2012spectra}
  J.~Cooper, A.~Dutle, Spectra of uniform hypergraphs, Linear Algebra and its
    applications 436~(9) (2012) 3268--3292.
  
  \bibitem{Khan2015}
  M.-u.-I. Khan, Y.-Z. Fan, On the spectral radius of a class of
    non-odd-bipartite even uniform hypergraphs, Linear Algebra Appl. 480 (2015)
    93--106.
  
  \bibitem{beardon2000iteration}
  A.~F. Beardon, Iteration of rational functions: Complex analytic dynamical
    systems, Vol. 132, Springer Science \& Business Media, 2000.
  
  \bibitem{brand1955sequence}
  L.~Brand, A sequence defined by a difference equation, The American
    Mathematical Monthly 62~(7) (1955) 489--492.
  
  \bibitem{marshall1979inequalities}
  A.~W. Marshall, I.~Olkin, B.~C. Arnold, Inequalities: theory of majorization
    and its applications, Vol. 143, Springer, 1979.
  
  \bibitem{Wang_2020}
  F.~Wang, H.~Shan, Z.~Wang, On some properties of the $\alpha$-spectral radius
    of the k-uniform hypergraph, Linear Algebra and its Applications 589 (2020)
    62--79.
  
  \end{thebibliography}

\end{document}